\documentclass[]{article}

\usepackage{graphicx}
\usepackage{amsmath}
\usepackage{amssymb}
\usepackage{amsthm} 
\usepackage{bm}
\usepackage{pxfonts}
\usepackage{enumerate}
\usepackage{color}
\usepackage{mathdots}
\usepackage{sectsty}
\usepackage[hidelinks]{hyperref}
\usepackage{tikz}
\usepackage{caption}
\usepackage{adjustbox}
\usepackage{bbold}

\sectionfont{\scshape\centering\fontsize{11}{14}\selectfont}
\subsectionfont{\scshape\fontsize{11}{14}\selectfont}
\usepackage{fancyhdr}

\newcommand\shorttitle{Characteristic Polynomials of Orthogonal and Symplectic Random Matrices, Jacobi Ensembles \& L-functions}
\newcommand\authors{M. A. Gunes}

\newtheorem{thm}{Theorem}[section]

\newtheorem{defn}[thm]{Definition}
\newtheorem{rmk}[thm]{Remark}
\newtheorem{prop}[thm]{Proposition}
\newtheorem{conjecture}[thm]{Conjecture}

\newtheorem*{claim*}{Claim}

\title{\large \bf Characteristic Polynomials of Orthogonal and Symplectic Random Matrices, Jacobi Ensembles \& L-functions}
\author{\small MUSTAFA ALPER GUNES}
\date{}

\begin{document}

\maketitle
\begin{abstract}
Starting from Montgomery's conjecture, there has been a substantial interest on the connections of random matrix theory and the theory of L-functions. In particular, moments of characteristic polynomials of random matrices have been considered in various works to estimate the asymptotics of  moments of L-function families. In this paper, we first consider joint moments of the characteristic polynomial of a symplectic random matrix and its second derivative. We obtain the asymptotics, along with a representation of the leading order coefficient in terms of the solution of a Painlevé equation. This gives us the conjectural asymptotics of the corresponding joint moments over families of Dirichlet L-functions. In doing so, we compute the asymptotics of a certain additive Jacobi statistic, which could be of independent interest in random matrix theory. Finally, we consider a slightly different type of joint moment that is the analogue of an average considered over $U(N)$ in various works before. We obtain the asymptotics and the leading order coefficient explicitly.
\end{abstract}

\tableofcontents
\section{Motivation}
Let $U$ be a matrix in $SO(2N)$ or $Sp(2N)$ sampled according to the Haar measure with eigenvalues $e^{\pm i\theta_1},\ldots, e^{\pm i\theta_N}$ and define the functions,
\begin{equation*}
   \psi_U (\theta) =\prod_{j=1}^N \left(1-e^{-i\left(\theta-\theta_j\right)}\right) \left(1-e^{-i\left(\theta+\theta_j\right)}\right), 
\end{equation*}
and
\begin{equation*}
    \mathcal{Z}_U (\theta) =  e^{\frac{iN}{2}(\theta+\pi)-i\sum_{k=1}^N\frac{\theta_k}{2}} \prod_{j=1}^N \left(1-e^{-i\left(\theta-\theta_j\right)}\right).
\end{equation*}
Note that $\psi_U$ is the characteristic polynomial of $U$, whereas $\mathcal{Z}_U$ is the so-called $\mathcal{Z}$-function corresponding to the eigenvalues on the upper part of the unit circle. In this paper, we will compute the $N\to \infty$ asymptotics of the corresponding quantities:
\begin{equation}
    \mathfrak{R}_N^G(s,h) =\mathbb{E}_{N}\left[\left| \mathcal{Z}_U (0)\right|^{2s} \left|\frac{d}{d\theta}\log \mathcal{Z}_U(\theta)\bigg\rvert_{\theta=0}\right|^{h}\right],
    \label{eq:average1}
\end{equation}
and
\begin{equation}
    \mathfrak{L}_N^G(s,h)=\mathbb{E}_N \left[\left|\psi_U(0)\right|^{2s-h} \left|\psi''_U(0)\right|^h\right],
    \label{eq:average2}
\end{equation}
where $G$ is one of the matrix groups $SO(2N)$ or $Sp(2N)$, and the averages are taken with respect to the Haar measures on each group.\\
While $\mathfrak{R}_N^G(s,h)$ can be seen as an analogue (notwithstanding certain differences) of a quantity that has been considered as an average over $U(N)$ in various works before due to its connections with the Riemann $\zeta$-function, (see \cite{Winn}, \cite{7authors}, \cite{abgs}, \cite{assiotis2020joint}, \cite{Bailey_2019}, \cite{Hughes}), results obtained in this paper regarding the asymptotics of $\mathfrak{L}_N^G(s,h)$ give rise to conjectures about moments of families of L-functions.
\subsection{Random Matrix Theory \& L-functions}
Starting from a conjecture due to Montgomery \cite{Montgomery}, there has been a substantial amount of work done on the connections between random matrix theory and analytic number theory. Montgomery's conjecture, which was based on both rigorous calculations and certain conjectures on twin primes, stated that the pair correlation function of the zeros of the Riemann Zeta function, high up the critical line, could be modeled via the pair correlation function of eigenangles of a random unitary matrix of large size. Later, it was proved by Keating and Snaith \cite{KeatingSnaith}, that the value distribution of $\zeta(s)$ coincided with the distribution of the characteristic polynomial of a random unitary matrix in the same limit. In \cite{KatzSarnak}, as an attempt to generalize this connection to other classes of L-functions, Katz and Sarnak  considered various statistics related to Haar-distributed random matrices from classical compact groups, $U(N)$, $Sp(2N)$ and $SO(2N)$. They showed that in the $N\to \infty$ limit, certain statistics related to consecutive spacings of eigenangles all converge to the same limit, that of the Circular Unitary Ensemble (CUE). On the other hand, they realized that this universality does not hold if one considers the distribution of the $k$-th eigenvalue over these compact groups. Namely, let $G$ be one of the groups, $U(N), Sp(2N)$ or $O(N)$ and let the eigenvalues of a matrix in these groups be ordered around the unit circle (starting from $1$ and going anti-clockwise). It is proved in \cite{KatzSarnak} that the averages
 \begin{equation*}
           \nu_{k} (G) [\alpha,\eta]:= \mu_{\textbf{Haar}} \left\{ A \in G: \alpha \leq \frac{\theta_k (A) N}{2 \pi} \leq \eta\right\},
       \end{equation*}
as $N\to \infty$, converge to different quantities corresponding to the symmetry type of the matrix group.\\
To translate this to the language of L-functions, consider a family of L-functions $\mathcal{L}$ where the L-functions are ordered by their conductors $c(L)$ (as we will see later on, they may be ordered by various other types of quantities, for instance cusp form L-functions are ordered by their weight).\\
If one assumes the Riemann Hypothesis for these L-functions, then all the zeros lie on the critical line (we will assume that the critical line is always $\Re(s)=\frac{1}{2}$) and the ones with non-negative imaginary part can be written as:
\begin{equation*}
    \left\{\frac{1}{2}+ i\gamma_n\right\}_{n=1}^\infty,
\end{equation*}
where $\gamma_n$'s are in increasing order.
Katz and Sarnak provide strong evidence that analogously to the random matrix theory average, as $N\to \infty$, the quantity
 \begin{equation*}
           \frac{ \left| \left\{L\in \mathcal{L}, c(L) \leq N  : \alpha \leq \frac{\gamma_{L,j} \log c(L)}{2\pi}\leq \eta\right\}\right|}{\left|
            \left\{L \in \mathcal{L}, c(L) \leq N\right\}\right|},
        \end{equation*}
converges where $\gamma_{L,j}$ denotes the $j$-th zero on the line $\Re(s)=\frac{1}{2}$. The factor $\log c(L)$ appears here due to the density of the zeros. The limiting quantity, as they conjecture, depends on the \textit{symmetry type} of the L-function family, which, in most cases, is difficult to determine.  \\
Due to the strong evidence that the zero statistics near $s=\frac{1}{2}$ over families of L-functions can be modeled by one of the compact random matrix groups, it becomes natural to wonder whether one can obtain conjectural asymptotics for a sum of the form
\begin{equation*}
    \frac{1}{|\mathcal{L}_N|} \sum_{L \in \mathcal{L}_N} L\left(\frac{1}{2}\right)^m,
\end{equation*}
as $N\to \infty$, where $\mathcal{L}$ is an L-function family ordered by some quantity $N$, and $\mathcal{L}_N$ is the subset of the family determined by $N$. Conrey and Farmer \cite{ConreyFarmer} provide evidence that asymptotically these averages are also determined by the symmetry type of the family. Hence, one would expect that these asymptotics agree with asymptotics of moments of characteristic polynomials of random matrices from the corresponding matrix group. Keating and Snaith give conjectures in this direction in \cite{keatingsnaith2}. Thus, we expect that the asymptotics we will compute give rise to conjectures regarding families of orthogonal and symplectic L-functions. Here note that we consider averages over $SO(2N)$ since L-functions with orthogonal symmetry either exhibit odd or even symmetry and in the odd case, the function vanishes at $s=\frac{1}{2}$ whereas for even symmetry we expect that the statistics related to the zeros near $s=\frac{1}{2}$ exhibit behaviour similar to eigenangles of $SO(2N)$ (see \cite{KatzSarnak} for any details).
\subsection{Main Results}
For technical reasons that appear in the proofs, we are able to prove our main results for a restricted parameter range (which still covers most of the range for which these moments exist). Thus, for ease of notation, if $x\in \mathbb{R}$, we will let $\alpha(x)$ be the greatest integer strictly less than $x$. In addition, throughout the rest of the paper, we will let $G(\cdot)$ denote the Barnes $G$-function. We are finally in position to state our main results. 
\begin{thm}
Let $\mathfrak{L}_N^{Sp}(s,h)$ be defined as before. For $\alpha\left(s+\frac{3}{2}\right)>h\geq 0$, we have the convergence
\begin{equation*}
    \lim_{N \to \infty} \frac{\mathfrak{L}_N^{Sp}(s,h)}{N^{\frac{s(s+1)}{2} + 2h} } =  \frac{2^{\frac{s^2}{2}}G(1+s)\sqrt{\Gamma(1+s)}}{\sqrt{G(1+2s)}\Gamma(1+2s)}\mathbb{E}\left[\left(\frac{1}{2}M\left(s+\frac{1}{2},\frac{1}{2}\right)+1\right)^h\right].
\end{equation*}
where for $a\geq b$, $a,b>0$, $M(a,b)$ is a non-trivial random variable whose distribution is given by \eqref{eq:limitingrvapprox}. Furthermore, the $\tau$-function corresponding to $M(a,b)$, which is defined as:
\begin{equation*}
      \tau_{a,b}(t):= t \frac{d}{dt} \log \mathbb{E}\left[e^{-tM(a,b)}\right],
\end{equation*}
solves the following form of the $\sigma$-Painlevé-III system:
\begin{equation*}
    \left(t \frac{d^2 \tau_{a,b} (t)}{dt^2}\right)^2+4 \left(\frac{d \tau_{a,b}}{dt}\right)^2 \left(t \frac{d\tau_{a,b}}{dt}-\tau_{a,b}(t)\right)-\left(a \frac{d\tau_{a,b}}{dt}+1\right)^2=0,
\end{equation*}
for all $t>0$, with boundary conditions: 
\begin{align}
\begin{cases}
    \tau_{a,b}(0) = 0, & \text{for } \ a>0, \label{eq:xiInfbc1}\\ & \\
    \left.\frac{d}{dt} \tau_{a,b}(t)\right|_{t=0} = -\frac{1}{4a}, & \text{for}\ a>1.
\end{cases}
\end{align}
\label{mainthm2}
\end{thm}
\begin{rmk}
We believe that the random variable $M(a,b)$ in fact only depends on $a$, and is given by the sum of inverse points of the Bessel point process with parameter $a$ (which we will denote by $Y(a)$). This would not be very surprising since the correlation kernel of the Jacobi ensemble converges to the Bessel correlation kernel under the $\frac{1}{N^2}$ scaling, and the Laplace transform of $Y(a)$ is also shown to satisfy the Painlevé system given above (see \cite{abgs}). Finally, we also note that this random variable appears in the ergodic decomposition of inverse Wishart measures (see \cite{assiotis}).
\end{rmk}
At the centre of our proofs will be the Jacobi ensemble, which we define now.
\begin{defn}
Let $a,b>-1$. Define the probability measure, supported on $[0,1]^N$:
\begin{equation}
    \mu_{N}^{(a,b)} (d\mathbf{x}) = \prod_{j=0}^{N-1} \frac{\Gamma(a+b+N+j+1)}{\Gamma(a+j+1)\Gamma(b+j+1)\Gamma(j+2)} \prod_{j=1}^N x_j^a \left(1-x_j\right)^b \prod_{1\leq j<k\leq N} \left(x_k-x_j\right)^2 d\mathbf{x}.
    \label{eq:jacobidef}
\end{equation}
This is called the Jacobi ensemble corresponding to parameters $a,b$. Similarly, for a measurable function $f:\mathbb{R}^N\to \mathbb{R}$, we denote
\begin{equation*}
    \mathbb{E}_N^{(a,b)}[f(x_1,\ldots,x_N)]:= \int_{\mathbb{R}^N} f(\mathbf{x}) \mu_N^{(a,b)}(d\mathbf{x}).
\end{equation*}
\end{defn}
In section 2, we will start by writing $\mathfrak{L}_N^{Sp}$ in terms of the moments of a linear Jacobi statistic. We will show that all derivatives of the Laplace transform of this statistic converge to those of a limiting random variable. In essence, the transformation to the Jacobi ensemble should work for $SO(2N)$ as well, but the Riemann-Hilbert methods used in \cite{chen} require an extra restriction on the parameters of the Jacobi ensemble. That is, it is crucial for the arguments in this paper that the Laplace transform of $\{M_N\}_{N\geq 1}$ converges pointwise to some limiting function. However, for the Jacobi parameters that correspond to $SO(2N)$, this does not follow from the results of \cite{chen}. Finally, we also note that the analogous linear statistic corresponding to the Laguerre unitary ensemble has been studied in various works before (see \cite{ref2_1}), \cite{ref2_2}, \cite{ref2_3}). \\
\indent This result, following the discussion in the previous section, gives rise to the following conjecture about joint moments of Dirichlet L-functions. These type of joint moments have been studied in the past both in random matrix theory and number theory settings (see, for instance, \cite{Hughes}, \cite{Mourtada}). Here we denote by $L(s,\chi_d)$, the L-function corresponding to a Dirichlet character modulo $d$, $\chi_d$.
\begin{conjecture}
Let $s+\frac{1}{2}>h \geq 0$. Then, as $N\to \infty$,
\begin{multline*}
  \frac{1}{N^*} \sum \left|L\left(\frac{1}{2},\chi_d\right)\right|^{2s-h} \left|L'' \left(\frac{1}{2},\chi_d\right)\right|^{h} \sim \log^{s(s+1)/2+2h} (N^{1/2})\\a_s \frac{2^{\frac{s^2}{2}}G(1+s)\sqrt{\Gamma(1+s)}}{\sqrt{G(1+2s)}\Gamma(1+2s)}\mathbb{E}\left[\left(\frac{1}{2}M\left(s+\frac{1}{2},\frac{1}{2}\right)+1\right)^h\right],
\end{multline*}
where for the sum is taken over all quadratic Dirichlet characters with modulus less than or equal to $N$, $N^*$ is the number of such characters and 
\begin{equation*}
    a_s = \prod_{p \text{ prime }} \frac{\left(1-\frac{1}{p}\right)^{s(s+1)/2}}{1+\frac{1}{p}}\left(\frac{1}{p}+\frac{1}{2}\left(\left(1+\frac{1}{\sqrt{p}}\right)^{-s}+\left(1-\frac{1}{\sqrt{p}}\right)^{-s}\right)\right).
\end{equation*}
\end{conjecture}
Using a similar method, that is switching to the Jacobi ensemble, one can also obtain the asymptotics of $\mathfrak{R}_N(s,h)$, which requires a few extra technical results that are proved using the theory of orthogonal polynomials. This is the content of our second main result. Here, $G$ denotes the Barnes G-function.
\begin{thm}
For $s>\frac{1}{2}$, $\alpha\left(2s+1\right)>h\geq 0$,
\begin{equation*}
    \lim_{N \to \infty} \frac{\mathfrak{R}_N^{SO}(s,h)}{N^{\frac{s(s-1)}{2} + h} \log^h (N)} = \frac{1}{\pi^h} \frac{2^{\frac{s^2}{2}}G(1+s)\sqrt{\Gamma(1+2s)}}{\sqrt{G(1+2s)}\Gamma(1+s)},
\end{equation*}
and for $s>\frac{3}{2}$, $\alpha(2s+3)>h\geq 0$,
\begin{equation*}
    \lim_{N \to \infty} \frac{\mathfrak{R}_N^{Sp}(s,h)}{N^{\frac{s(s+1)}{2} + h} \log^h (N)} = \frac{1}{\pi^h} \frac{2^{\frac{s^2}{2}}G(1+s)\sqrt{\Gamma(1+s)}}{\sqrt{G(1+2s)}\Gamma(1+2s)}.
\end{equation*}
\label{mainthm}
\end{thm}
\textbf{Acknowledgements.} We would like to thank Prof. Jon Keating for many useful suggestions that lead to the improvement of both the results and their presentation. We would also like to thank anonymous referees both for their very helpful comments and suggestions that improved the presentation of the paper, and also for pointing out a number of important references.
\section{Proof of Theorem \ref{mainthm2}}

The relation between the characteristic polynomials of symplectic/orthogonal matrices and Jacobi ensemble is due to Weyl's integration formula. In particular, via a standard application of Weyl's formula one can compute (see for instance \cite{Forrester}) the eigenangle densities of these compact matrix groups:
\begin{center}
\begin{tabular}{c c c c} 
 \hline
 G & Eigenvalues & Joint pdf of $\theta$s \\ 
 \hline\hline

 \hline
 $SO(2N)$ & $e^{\pm i \theta_1}, \ldots, e^{\pm i \theta_N}$ & $ \prod\limits_{1 \leq k <j \leq N} \left|\cos\left(\theta_k\right)-\cos\left(\theta_j\right)\right|^2 d \bm{\theta}$   \\
 \hline
 $Sp(2N)$ & $e^{\pm i \theta_1}, \ldots, e^{\pm i \theta_N}$ & $ \prod\limits_{1 \leq k <j \leq N} \left|\cos\left(\theta_k\right)-\cos\left(\theta_j\right)\right|^2  \prod\limits_{k=1}^N \sin^2 (\theta_k)d \bm{\theta}$ \\
 
 \hline
\end{tabular}
\end{center}
Using the above table, one can see that the transformation $x_j=\frac{1}{2}(1-\cos \theta_j)$ transfers these measures to the corresponding Jacobi ensembles. \\
Now, focusing again on the average $\mathfrak{L}_N^{Sp}(s,h)$, we compute
\begin{equation*}
   \frac{d^2 \psi_U}{d\theta^2} = \sum_{j=1}^N \left\{\psi''_j(\theta) \prod_{\substack{l=1 \\ l \neq j}}^N \psi_j(\theta) + \psi'_j(\theta) \sum_{\substack{l=1 \\ l \neq j}}^N \left\{ \psi'_l(\theta) \prod_{\substack{k=1 \\ k \neq l,j}}^N \psi_k(\theta)\right\}\right\},
\end{equation*}
where
\begin{equation*}
    \psi_j(\theta)=(1-e^{i(\theta-\theta_j})(1-e^{-i(\theta-\theta_j)}).
\end{equation*}
Dividing by $\psi_U$ and evaluating at $\theta=0$, we obtain;
\begin{equation*}
    \frac{\psi''_U(0)}{\psi_U(0)}=\sum_{j=1}^N \frac{\psi''_j(0)}{\psi_j(0)} + \sum_{\substack{j,l=1 \\ j \neq l}}^N \frac{\psi'_j(0)\psi'_l(0)}{\psi_j(0)\psi_l(0)}=  \sum_{j=1}^N \frac{\cos(\theta_j)-2}{1-\cos(\theta_j)} - N(N-1).
\end{equation*}
In particular, applying the transformation $x_j=\frac{1}{2}\left(1-\cos(\theta_j)\right)$:
\begin{equation}
    \mathfrak{L}_N^{Sp}(s,h) = \mathfrak{L}_N^{Sp}(s,0) \mathbb{E}_{N}^{\left(s+\frac{1}{2},+\frac{1}{2}\right)} \left[\left(\sum_{j=1}^N \frac{1}{2x_j}+N^2\right)^h\right].
\end{equation}
The $N\to \infty$ asymptotics of $\mathfrak{L}_N^{Sp}(s,0)$ is already given in \cite{keatingsnaith2} by
\begin{equation}\label{eq:0thmomasymptoticssp}
    \mathfrak{L}_N^{Sp}(s,0)\sim N^{s(s+1)/2} \frac{2^{\frac{s^2}{2}}G(1+s)\sqrt{\Gamma(1+s)}}{\sqrt{G(1+2s)}\Gamma(1+2s)}.
\end{equation}
Hence, the following theorem, which is the main result of this section, gives the asymptotics of $\mathfrak{L}_N^{Sp}(s,h)$.
\begin{thm}

Let $a\geq b$, $a,b>0$. If $\left(x_1^{(N)},\ldots, x_N^{(N)}\right)$ denote points sampled according to $\mu_N^{\left(a,b\right)}\left(d\bm{x}\right)$, there exists a non-negative (and non-trivial) random variable $M(a,b)$ such that
\begin{equation}\label{eq:limitingrvapprox}
    M_N:= \sum_{j=1}^N \frac{1}{N^2x_j^{(N)}} \xrightarrow[N \to \infty]{d} M(a,b),
\end{equation}
and for all $h\in[0,\alpha\left(a+1)\right)$
\begin{equation*}
 \mathbb{E}_{N}^{\left(a,b\right)} \left[\left(\frac{M_N}{2}+1\right)^h\right] \xrightarrow[]{N \to \infty}  \mathbb{E} \left[\left(\frac{M(a,b)}{2}+1\right)^h\right].
\end{equation*}
Furthermore, the 'tau' function corresponding to the random variable $M(a,b)$,
\begin{equation*}
      \tau_{a,b}(t):= t \frac{d}{dt} \log \mathbb{E}\left[e^{-tM(a,b)}\right],
\end{equation*}
solves the following form of the $\sigma$-Painlevé-III equation:
\begin{equation*}
    \left(t \frac{d^2 \tau_{a,b} (t)}{dt^2}\right)^2+4 \left(\frac{d \tau_{a,b}}{dt}\right)^2 \left(t \frac{d\tau_{a,b}}{dt}-\tau_{a,b}(t)\right)-\left(a \frac{d\tau_{a,b}}{dt}+1\right)^2=0.
\end{equation*}
\end{thm}
To prove this result, we will first prove convergence in distribution, and then show that the $h$-th moments are uniformly integrable so that convergence of $L^h$ norms also takes place. In order to do this, we first focus on $h\in \mathbb{N}$. Throughout the rest of this section, whenever we say $M_N$ \textit{sampled according to $\mu_N^{(a,b)}$}, we will mean $M_N$ is defined as in \eqref{eq:limitingrvapprox} where $(x_1^{(N)},\ldots, x_N^{(N)})$ has the law 
$\mu_N^{(a,b)}$. Now, notice that one may expand:
\begin{equation}
    \mathbb{E}_N^{(a,b)}\left[M_N^h\right] = \frac{1}{N^{2h}} \sum_{\substack{\lambda_1+\ldots+\lambda_l=h\\ \lambda_1 \geq \ldots \geq \lambda_l>0}} \binom{h}{\lambda_1, \ldots, \lambda_l} \kappa_N(\lambda_1, \ldots, \lambda_l)\mathbb{E}_{N}^{\left(a,b\right)}\left[\prod_{j=1}^l x_j^{-\lambda_j}\right],
    \label{eq:expansion}
\end{equation}
where
\begin{equation*}
    \kappa_N(\lambda_1,\ldots,\lambda_l)=\frac{N!}{(N-l)! n_1!\ldots n_m!}.
\end{equation*}
Here, $m$ is the number of distinct integers occurring in the partition $(\lambda_1,\ldots,\lambda_l)$ and $n_k$ is the number of times $k$-th integer occurs. Notice that to obtain this sum, we expanded the $h$-th power as a multinomial and used the symmetry of the Jacobi ensemble under permutations of indices. \\
Using standard arguments from random matrix theory (see, for instance, \cite{Forrester}), one can show that the $l$-point density of the Jacobi ensemble in \eqref{eq:jacobidef} is given by:
\begin{equation}
    \rho_{N}^{(a,b)}\left(x_1,\ldots, x_l\right)= \frac{(N-l)!}{N!} \det\left[K_N^{(a,b)}(x_i,x_j)\right]_{i,j=1}^l.
    \label{eq:dens}
\end{equation}
The so-called $\textit{correlation kernel}$, $K_N(x,y)$, appearing in the determinant is given by:
\begin{equation*}
    K_N^{(a,b)}(x,y) = \sqrt{w^{(a,b)}(x) w^{(a,b)}(y)} \sum_{j=1}^N P_j^{(a,b)}(x) P_j^{(a,b)}(y),
\end{equation*}
where 
\begin{equation*}
    w^{(a,b)}(x) = x^a(1-x)^b,
\end{equation*}
and $\left\{P_j^{(a,b)}\right\}_{j=1}^\infty$ is the sequence of orthonormal Jacobi polynomials corresponding to parameters $a,b$. To be more precise, these are polynomials indexed by their degree that satisfy the orthonormality condition:
\begin{equation*}
    \int_0^1 P_j^{(a,b)}(x) P_k^{(a,b)} (x) w^{(a,b)} (x) dx = \delta_{jk}.
\end{equation*}
In particular, by studying the asymptotics of fixed-dimensional integrals of the correlation kernel, one should be able to obtain the asymptotics of moments of $M_N(a,b)$.
\begin{prop}
Let $a\geq b \geq -\frac{1}{2}$, $h \in \mathbb{N}$, $0<h<a+1$, and  $(\lambda_1,\ldots,\lambda_l)$ be a partition of $h$. Then;
\begin{equation*}
   \kappa_N(\lambda_1, \ldots, \lambda_l)\mathbb{E}_{N}^{\left(a,b\right)}\left[\prod_{j=1}^l x_j^{-\lambda_j}\right] = O \left(N^{2h}\right).
\end{equation*}
\label{pro1}
\end{prop}
\begin{proof}
Using \eqref{eq:dens}, (from here on the constant may change from line to line)

\begin{equation}
     \kappa_N(\lambda_1, \ldots, \lambda_l)\mathbb{E}_{N}^{\left(a,b\right)}\left[\prod_{j=1}^l x_j^{-\lambda_j}\right] \leq  \text{const } \int_{[0,1]^l} \prod_{j=1}^l x_j^{-\lambda_j} \det\left[K_N^{(a,b)}(x_i,x_j)\right]_{i,j=1}^l dx_1 \ldots dx_l.
     \label{eq:smallint2}
\end{equation}
Now exapnding the determinant as a finite sum over the $l$-th symmetric group, and using the Cauchy-Schwarz inequality to note
\begin{equation}
    \left|K_N^{\left(a,b\right)}(x,y)\right| \leq \sqrt{K_N^{\left(a,b\right)}(x,x)} \sqrt{K_N^{\left(a,b\right)}(y,y)},
\end{equation} 
we can bound the integral in \eqref{eq:smallint2} by:
\begin{equation*}
    \text{const} \cdot \prod_{j=1}^l \int_0^1 x^{-\lambda_j} K_N^{\left(a,b\right)}(x,x) dx .
\end{equation*}
We claim, for $\lambda\geq 1$, that 
\begin{equation}
    \int_0^1 x^{-\lambda} K_N^{\left(a,b\right)}(x,x) dx = O\left(N^{2\lambda} \right) .
    \label{eq:momentsdct}
\end{equation}
First, note that we may restrict the integration to $(0,\frac{1}{2})$ since
\begin{equation*}
 \int_{\frac{1}{2}}^1 x^{-\lambda} K_N^{\left(a,b\right)}(x,x) dx \leq 2^{h} \int_{\frac{1}{2}}^1 K_N^{(a,b)}(x,x)dx =2^h N\mathbb{P}\left(x_1\in \left(\frac{1}{2},1\right)\right) = O(N^{2\lambda}).
\end{equation*}
Now, making the substitution $x= \frac{u}{N^2}$, the integral over $\left(0,\frac{1}{2}\right)$ becomes:
\begin{equation*}
    N^{2\lambda} \int_0^{N^2/2
    } u^{-\lambda}  \frac{1}{N^2} K_N^{\left(a,b\right)}\left(\frac{u}{N^2},\frac{u}{N^2}\right) du.
\end{equation*} 
On $(0,1)$, we use Szegö's well-known bound (see \cite{szego}):
\begin{equation*}
    \sup_{x\in [0,1]} P_j^2(x) \leq \text{const} \cdot j^{2a+1},
\end{equation*}
to see that
\begin{multline*}
     \int_0^1 u^{-\lambda}  \frac{1}{N^2} K_N^{\left(a,b\right)}\left(\frac{u}{N^2},\frac{u}{N^2}\right) du \\ \leq \frac{\text{const}}{N^{2a+2}} \sum_{j=0}^N  j^{2a+1}\int_0^1 u^{a-\lambda} du \leq \text{const} \cdot \frac{1}{N^{2a+2}} \sum_{j=0}^N j^{2a+1} = O(1) .
\end{multline*}
On $(1,\frac{N^2}{2})$, we use the bound (see \cite{Erdeyli}):
\begin{equation*}
    w^{\left(a,b\right)} (x) \left(P_j^{\left(a,b\right)} (x)\right)^2 \leq \frac{\text{const}}{\sqrt{x(1-x)}},
\end{equation*}
to note that
\begin{equation*}
     \int_1^{N^2/2
    } u^{-\lambda}  \frac{1}{N^2} K_N^{\left(a,b\right)}\left(\frac{u}{N^2},\frac{u}{N^2}\right) du \leq \text{const} \int_1^{N^2} u^{-\lambda-\frac{1}{2}} du = O(1),
\end{equation*}
which gives us the desired result.
\end{proof}
\begin{prop}
Let $a\geq b\geq -\frac{1}{2}$, $a>0$. Then, the sequence $\left\{M_N\right\}_{N\geq 1}$sampled according to $\mu_N^{(a,b)}$ is tight.
\label{prop:tight}
\end{prop}
\begin{proof}
Using Markov's inequality:
\begin{equation*}
    \sup_{N \in \mathbb{N}} \mathbb{P}_N^{(a,b)}\left(M_N \geq R\right) \leq \sup_{N \in \mathbb{N}} \frac{\mathbb{E}_N^{(a,b)}\left[M_N\right]}{R}.
\end{equation*}
We claim that
\begin{equation*}
    \sup_{N \in \mathbb{N}}\mathbb{E}_N^{(a,b)}\left[M_N\right] < \infty.
\end{equation*}
Indeed, using symmetry:
\begin{equation*}
    \mathbb{E}_N^{(a,b)}\left[M_N\right] = \frac{1}{N} \mathbb{E}_N^{(a,b)} \left[\frac{1}{x_1}\right]= \frac{1}{N^2} \int_0^1 \frac{K_N^{\left(a,b\right)}(x,x)}{x} dx.
\end{equation*}
We first restrict the integral to $[0,\frac{1}{2}]$ by using the boundedness of the random variable on $\left[\frac{1}{2},1\right]$. Then, arguing as in the proof of Proposition \ref{pro1}, we see that this integral is bounded uniformly in $N$. Hence,
\begin{equation*}
    \lim_{R\to \infty} \sup_{N \in \mathbb{N}} \mathbb{P}_N^{(a,b)}\left(M_N \geq R\right) = 0,
\end{equation*}
which proves the proposition.
\end{proof}
\begin{prop}
For $a\geq b$; $a,b>0$ and $h \in [0,\alpha\left(a+1)\right)$, the sequence $\left\{M_N^h\right\}_{N\geq 1}$ where $M_N$ is sampled according to $\mu_N^{(a,b)}$ is uniformly integrable.
\end{prop}
\begin{proof}
It is sufficient to prove that (let $\alpha=\alpha\left(a+1\right)$) 
\begin{equation*}
    \sup_{N \in \mathbb{N}} \mathbb{E}_N^{(a,b)}\left[M_N^\alpha\right] < \infty.
\end{equation*}
To see this, we may expand as before:
\begin{equation}
    \mathbb{E}_N^{(a,b)}\left[M_N^\alpha\right] = \frac{1}{N^{2\alpha}} \sum_{\substack{\lambda_1+\ldots+\lambda_l=\alpha\\ \lambda_1 \geq \ldots \geq \lambda_l>0}} \binom{\alpha}{\lambda_1, \ldots, \lambda_l} \kappa_N(\lambda_1, \ldots, \lambda_l)\mathbb{E}_{N}^{\left(a,b\right)}\left[\prod_{j=1}^l x_j^{-\lambda_j}\right].
\end{equation}
Using the bounds obtained in Proposition \ref{pro1}, this is uniformly bounded in $N$, which completes the proof, since uniform boundedness of a higher moments of a random variable gives the uniform integrability of lower moments.
\end{proof}
\begin{prop}\label{prop:convergence}
Fix $a\geq b$; $a,b>0$. Then there exists a holomorphic function $f(z)$ on $S:=\left\{z \in \mathbb{C}: \Re(z) > 0 \right\}$ such that uniformly on compact subsets of $S$,
\begin{equation*}
    f_N(z) := \mathbb{E}_N^{(a,b)}\left[e^{-zM_N}\right] \xrightarrow[]{\text{as } N \to \infty  } f(z).
\end{equation*}
\end{prop}
\begin{proof}
First, note that $f_N(t)$ converges pointwise for all $t\in (0,\infty)$ (see  \cite{chen}). Now, if $z\in{S}$,
\begin{equation*}
    \left|\mathbb{E}_N^{(a,b)}\left[e^{-zM_N}\right]\right| \leq \mathbb{E}_N^{(a,b)}\left[\left|e^{-zM_N}\right|\right] \leq \mathbb{E}_N^{(a,b)}[1] =1,
\end{equation*}
and hence by Montel's theorem, every subsequence of $f_N$ has a further subsequence that converges uniformly on compact subsets of $S$ to a holomorphic function. Further, these limiting holomorphic functions agree on $(0,\infty)$ so that by identity theorem, every converging subsequence converges uniformly on compacts to the same limiting function, $f(z)$. Hence, the full sequence converges to $f(z)$ uniformly on compacts. 

\end{proof}
\begin{proof}[Proof of Theorem \ref{mainthm2}]
By Proposition \ref{prop:tight}, the sequence $\{M_N\}_{N\geq 1}$ is tight so that there exists a subsequence $\{M_{N_k}\}_{k \geq 1}$ that converges weakly to some limiting random variable $M=M(a,b)$. In particular, for all $t \in (0,\infty)$,
\begin{equation}
    \lim_{k \to \infty}f_{N_k}(t) =f(t)= \mathbb{E}\left[e^{-tM(a,b)}\right].
\end{equation}
But then by Proposition \ref{prop:convergence}, this implies,
\begin{equation*}
    \lim_{N\to \infty} f_N(t) = \mathbb{E}
    \left[e^{-tM(a,b)}\right].
\end{equation*}
Thus, using a standard result from probability theory (see for example \cite{foundations}), we deduce that 
\begin{equation*}
    M_N \xrightarrow[d]{\text{as } N \to \infty  } M=M(a,b).
\end{equation*}
which proves the first claim of the theorem. To prove the second claim, we simply use Skorokhod's representation theorem so that we may assume without loss of generality that the convergence is in fact in almost sure sense. Combining this with the uniform integrability of $\{M_N^h\}$ gives the second part of the claim. For the final part, we use Proposition \ref{prop:convergence} to observe that since $\{f_N\}_{N\geq1}$ converges uniformly on compact sets to $f$, so do all the derivatives of  $\{f_N\}_{N\geq1}$. Hence, one may take the limit of the finite-$N$ differential equation (and the corresponding boundary conditions) proved for $f_N(t)$ in \cite{chen} to get the desired differential equation.\footnote{We note that the scaling used in \cite{chen} has an extra factor of $2$, which we ommitted here to make the presentation clearer, which causes a minor difference in the differential equation and boundary conditions.} For the boundary conditions, we note that the extra conditions on the parameters is due to the fact that for $\tau_{a,b}$ to have derivatives at $0$, we need certain moments of $M(a,b)$ to be finite.
\end{proof}
\section{Proof of Theorem \ref{mainthm}}

In this section, we will compute the $N\to \infty$ asymptotics of $\mathfrak{R}_N^G(s,h)$. We start by observing that using the same transformation as before,
$x_j =\frac{1}{2}\left(1- \cos \theta_j\right)$, one may write, 
\begin{equation}
   \mathfrak{R}_N^G(s,h)= \frac{\mathfrak{R}_N^G(s,0)}{2^h} \cdot \mathbb{E}_{N}^{\left(a,b\right)} \left[\left|\sum_{i=1}^N  \frac{\sqrt{1-x_j}}{\sqrt{x_j}}\right|^{h}\right],
     \label{eq:jacobitr}
\end{equation}
where $(a,b)=(s-\frac{1}{2},-\frac{1}{2})$ if $G=SO(2N)$; and $(a,b)=(s+\frac{1}{2},\frac{1}{2})$ if $G=Sp(2N)$. Here, as before, $\mathbb{E}_{N}^{\left(a,b\right)}$ denotes expectation taken against the Jacobi ensemble with parameters $a,b$. Now, the asymptotics of $\mathfrak{R}_N^{Sp}(s,0)=\mathfrak{L}_N^{Sp}(s,0)$ are given in \eqref{eq:0thmomasymptoticssp}; and the asymptotics $\mathfrak{R}_N^{SO}(s,0)$ are also computed explicitly in \cite{keatingsnaith2} to be:
\begin{equation*}
    \mathfrak{R}_N^{SO}(s,0)\sim N^{\frac{s(s-1)}{2}} \frac{2^{\frac{s^2}{2}}G(1+s)\sqrt{\Gamma(1+2s)}}{\sqrt{G(1+2s)}\Gamma(1+s)}.
\end{equation*}
Hence, it becomes sufficient to compute the $N\to \infty$ asymptotics of
\begin{equation*}
   \mathbb{E}_{N}^{\left(a,b\right)} \left[\left|\sum_{i=1}^N  \frac{\sqrt{1-x_j}}{\sqrt{x_j}}\right|^{h}\right],
\end{equation*}
for appropriate choices of $a,b$; which is precisely the content of the main result of this section.
\begin{thm}
Let $a\geq b\geq-\frac{1}{2}$ and $a> 0$. Then for all $h \in [0,\alpha(2a+2))$, we have the convergence:
\begin{equation*}
    \mathbb{E}_N^{(a,b)} \left[\left(\mathcal{Z}_N\right)^h\right]=\mathbb{E}_N^{(a,b)} \left[\left(\frac{1}{N \log N}\sum_{j=1}^N \frac{\sqrt{1-x_j^{(N)}}}{\sqrt{x_j^{(N)}}}\right)^h\right] \xrightarrow[]{ N \to \infty} \left(\frac{2}{\pi}\right)^{h}.
\end{equation*}

\end{thm}
We will first show the convergence of $\mathcal{Z}_N$ to $\frac{2}{\pi}$ in $L^2$, and then combine this with the uniform integrability of the $h$-th power of $\mathcal{Z}_N$. With this goal in mind, let us focus on the case $h \in \mathbb{N}$ for now. Then, as in equation \eqref{eq:expansion}, we can expand:
\begin{multline}
     \mathbb{E}_{N}^{\left(a,b\right)} \left[\left|\sum_{j
     =1}^N  \frac{\sqrt{1-x_j}}{\sqrt{x_j}}\right|^{h}\right] \\=  \sum_{\substack{\lambda_1+\ldots+\lambda_l=h\\ \lambda_1 \geq \ldots \geq \lambda_l>0}} \binom{h}{\lambda_1, \ldots, \lambda_l} \kappa_N(\lambda_1, \ldots, \lambda_l)\mathbb{E}_{N}^{\left(a,b\right)}\left[\prod_{j=1}^l x_j^{-\lambda_j/2}\left(1-x_j\right)^{\lambda_j/2}\right].
     \label{eq:combsum5}
\end{multline}
The main idea will be to show that the dominating contribution in this finite sum comes from the partition $\lambda_1=\ldots=\lambda_h=1$. Then, to compute the asymptotics of the dominating term for the second moment, we will write the expectation as a $2$-dimensional integral which will involve a $2\times 2$ determinant of the correlation kernel. We will see that the dominating contribution within this integral will come form the product of the diagonal entries. Thus, in essence, the main contribution will come from the integral whose asymptotics is computed in the following proposition.
\begin{prop}
Let $a\geq b \geq -\frac{1}{2}$, $a \geq 0$. We have:
\begin{equation*}
  \mathcal{S}_N=  \int_0^1 x^{-\frac{1}{2}}(1-x)^{1/2} K_N^{\left(a,b\right)}(x,x) dx \sim \frac{2}{\pi} N \log N.
\end{equation*}
\label{prop1}
\end{prop}

\begin{proof}
We start by fixing an $\epsilon \in (0,1)$ and dividing the integral into two contributions: $[0,\epsilon]$ and $[\epsilon,1]$. First, observe that:
\begin{multline}
    \frac{1}{N \log N}\int_\epsilon^1 x^{-\frac{1}{2}}(1-x)^{1/2} K_N^{\left(a,b\right)} (x,x) dx = \frac{1}{\log N} \mathbb{E}_{N}^{\left(a,b\right)} \left[\frac{\sqrt{1-x_1}}{\sqrt{x_1}}\mathbf{1}_{\{x_1\geq \epsilon\}}\right]\\ \leq \frac{1}{\log N} \frac{\sqrt{1-\epsilon}}{\sqrt{\epsilon}} \xrightarrow[]{\text{as } N \to \infty  } 0.
    \label{eq:zerocontribution}
\end{multline}
To compute the asymptotics of the contribution on $[0,\epsilon]$, one first needs to consider the $ l \to \infty$ asymptotics of:
\begin{equation}
   H(l):= \int_0^{\epsilon} x^{-\frac{1}{2}} (1-x)^{\frac{1}{2}} w^{\left(s-\frac{1}{2},-\frac{1}{2}\right)}(x) \left(P_l^{\left(a,b\right)}(x)\right)^2 dx.
\end{equation}
$\textbf{Step 1}:$ Let
\begin{equation*}
     F(l):=\int_0^\epsilon {h_l}\frac{\arccos\left(1-2x\right)}{x} \frac{J_a^2\left(\left(l+\frac{a+b+1}{2}\right)\arccos(1-2x)\right)}{\left(l+\frac{a+b+1}{2}\right)^{2a}} dx.
\end{equation*}
We claim that there exists a constant $\mathfrak{C}>0$ such that:
\begin{equation*}
    \sup_{l\in \mathbb{N}} \left|H(l)-F(l)\right| \leq \mathfrak{C}.
\end{equation*}
To prove this claim, we will use an asymptotic expansion obtained in \cite{FrenzenWong}:
\begin{multline}
    \sqrt{w^{\left(a,b\right)}(x)} P_l^{\left(a,b\right)}(x) \\=  \sqrt{h_l}\left(\frac{\arccos\left(1-2x\right)}{\sqrt{x(1-x)}}\right)^{\frac{1}{2}}  \left( \frac{J_{a}\left(\left(l+\frac{a+b+1}{2}\right)\arccos(1-2x)\right)}{\left(l+\frac{a+b+1}{2}\right)^{a}} + O\left(\frac{1}{\left(l+\frac{a+b+1}{2}\right)^{a+1}}\right)\right),
    \label{eq:wongest}
\end{multline}
where 
\begin{equation*}
    h_l = \frac{(2l+a+b+1)\Gamma \left(l+a+1\right)\Gamma(l+a+b+1)}{2\Gamma(l+1)\Gamma(l+b+1)},
\end{equation*}
and $J_\nu$ denotes the Bessel function with parameter $\nu$. Here, error term is uniform on $x \in [0,\epsilon]$. Now, note that for two sequences of functions on $[0,\epsilon]$; $\left\{a_l(x)\right\}_{l=1}^\infty$ and $\left\{b_l(x)\right\}_{l=1}^\infty$, if
\begin{equation*}
    \lim_{l \to \infty} \sup_{x \in [0,\epsilon]}\left|a_l(x)-b_l(x)\right|=0,
\end{equation*}
then for $l$ large enough,
\begin{equation*}
    \left|a_l^2(x)- b_l^2(x)\right| \leq \left|a_l(x)+b_l(x)\right| \leq 2(1 + \left|b_l(x)\right|).
\end{equation*}
Thus, using \eqref{eq:wongest}, and the asymptotics of $\Gamma(z)$; we can get the estimate:
\begin{multline*}
    x^{-\frac{1}{2}} (1-x)^{\frac{1}{2}} \biggl|w^{\left(a,b\right)}(x) \left(P_l^{\left(a,b\right)}(x)\right)^2  -  {h_l}\frac{\arccos\left(1-2x\right)}{\sqrt{x(1-x)}} \frac{J_{a}^2\left(\left(l+\frac{a+b+1}{2}\right)\arccos(1-2x)\right)}{\left(l+\frac{a+b+1}{2}\right)^{2a}} \biggl| \\ \leq \text{const}\cdot x^{-\frac{1}{2}} (1-x)^{\frac{1}{2}}\left(1+   \sqrt{l} \left(\frac{\arccos\left(1-2x\right)}{\sqrt{x(1-x)}}\right)^\frac{1}{2} \left|J_{a}\left(\left(l+\frac{a+b+1}{2}\right)\arccos(1-2x)\right)\right| \right)\\ \leq \text{const} \cdot   x^{-\frac{3}{4}},
\end{multline*}
where the constant is independent of $x$, and for the last inequality we use the bound,
\begin{equation*}
    \left|J_\alpha(x)\right| \leq C x^{-\frac{1}{2}}.
\end{equation*}
Since $x^{-3/4} \in L^1[0,\epsilon]$, our claim follows. \\
Now that our claim is established, we want to compute the asymptotics of $\sum_l H(l)$ by comparing with $\sum_l F(l)$. In particular, note that:
\begin{equation}
    \lim_{N\to \infty} \frac{1}{N \log N} \sum_{l=1}^N \left|H(l) -F(l)\right| \leq   \lim_{N\to \infty} \frac{1}{N \log N} \sum_{l=1}^N \text{const}=0.
    \label{eq:limiteq}
\end{equation}
\textbf{Step 2:} We compute the asymptotics of $\sum_l F(l)$. We begin by making the change of variables $x=\frac{1}{2}\left(1-\cos \theta\right)$ to obtain:
\begin{equation*}
    F(l) =\frac{h_l}{\left(l+\frac{a+b+1}{2}\right)^{2a}} \int_0^{\arccos(1-2\epsilon)}  \frac{\theta \sin \theta}{1-\cos \theta} J_{a}^2\left(\left(l+\frac{a+b+1}{2}\right)\theta\right) d\theta.
\end{equation*}
 Noting the fact that the function
\begin{equation}
    g(\theta): \theta \mapsto \frac{\theta \sin \theta}{1-\cos \theta},
\end{equation}
is decreasing on $[0,\epsilon]$ with $g(0)=2$; and making the change of variables $u=\left(l+\frac{a+b+1}{2}\right)\theta$, one can infer,
\begin{multline*}
   g(\arccos(1-2\epsilon)) \frac{h_l}{\left(l+\frac{a+b+1}{2}\right)^{2a+1}} \int_0^{\left(l+\frac{a+b+1}{2}\right)\arccos(1-2\epsilon)}  J_{a}^2(u) du\\ \leq  F(l)\leq 2 \frac{h_l}{\left(l+\frac{a+b+1}{2}\right)^{2a+1}} \int_0^{\left(l+\frac{a+b+1}{2}\right)\arccos(1-2\epsilon)}  J_{a}^2(u) du.
\end{multline*}
Now, on $[0,1]$, the integrand is bounded and hence we want to compute the growth at infinity. Using the well-known asymptotics of Bessel functions with large argument:
\begin{equation}
    J_a(x)= \sqrt{\frac{2}{\pi x}}\cos\left(x- \frac{\pi}{2}\left(a+\frac{1}{2}\right)\right) + O(x^{-3/2}),
    \label{eq:besselasymptotics}
\end{equation}
we obtain (where we let $l'=\left(l+\frac{a+b+1}{2}\right)\arccos(1-2\epsilon)$):
\begin{multline*}
    \int_1^{l'}  J_{a}^2(u) du =  \int_1^{l'} \frac{2}{\pi u} \cos^2\left(u- \frac{\pi}{2}\left(a+\frac{1}{2}\right)\right)    du +O(1)\\=\int_1^{l'} \frac{1}{\pi u} \left(1+\cos\left(2u- {\pi}\left(a+\frac{1}{2}\right)\right)\right)  du + O(1) = \frac{\log l}{\pi} + O(1).
\end{multline*}
Here, for the last equality we used the fact that
\begin{equation*}
    \sup_{l' \geq 1} \int_1^{l'} \frac{\cos\left(2u- {\pi}\left(a+\frac{1}{2}\right)\right)}{u} du < \infty.
\end{equation*}
Using Stirling's approximation for the Gamma function, one can compute: 
\begin{equation*}
    \lim_{l \to \infty} \frac{h_l}{\left(l+\frac{a+b+1}{2}\right)^{2a+1}}=1,
\end{equation*}
and hence if $\delta>0$, there exists $M \in \mathbb{N}$ such that for every $l \geq M$,
\begin{equation*}
    \frac{h_l}{\left(l+\frac{a+b+1}{2}\right)^{2a+1}} \leq 1+\delta.
\end{equation*}
In particular,
\begin{equation*}
   \limsup_{N\to \infty} \frac{1}{N \log N} \sum_{l=1}^N F(l)\leq  \limsup_{N \to \infty} \frac{2(1+\delta)\log \left(N!\right)}{\pi N \log N} + \frac{O(1)}{\log N} = \frac{2(1+\delta)}{\pi},
\end{equation*}
so that letting $\delta \to 0$, we get:
\begin{equation*}
    \limsup_{N\to \infty} \frac{1}{N \log N} \sum_{l=1}^N F(l)\leq  \frac{2}{\pi}.
\end{equation*}
Similarly one may obtain the lower bound:
\begin{equation*}
     \liminf_{N\to \infty} \frac{1}{N \log N} \sum_{l=1}^N F(l)\geq  \frac{g(\arccos(1-2\epsilon))}{\pi}.
\end{equation*}
But then by combining \eqref{eq:zerocontribution} and \eqref{eq:limiteq}, this implies:
\begin{equation*}
  \frac{g(\arccos(1-2\epsilon))}{\pi} \leq  \liminf_{N\to \infty} \mathcal{S}_N \leq  \limsup_{N\to \infty} \mathcal{S}_N \leq  \frac{2}{\pi}.
\end{equation*}
Since this holds for any $\epsilon >0$ small enough, we get the desired result.
\end{proof}
Thus, focusing on $h=2$, we expect the contribution coming from the partition $\lambda_1=\lambda_2=1$ to be of order $N^2 \log^2 N$. Next, we prove that the contribution coming from other terms have strictly smaller order than this.
\begin{prop}
Let $a\geq b \geq -\frac{1}{2}$, $h \in \mathbb{N}$, $0<l<h<2a+2$, and  $\lambda_1+\ldots+\lambda_l=h$. Then;
\begin{equation*}
   \kappa_N(\lambda_1, \ldots, \lambda_l)\mathbb{E}_{N}^{\left(a,b\right)}\left[\prod_{j=1}^l x_j^{-\lambda_j/2}\left(1-x_j\right)^{\lambda_j/2}\right] = O \left(N^h \log^{h-2} N \right).
\end{equation*}
\label{prop2}
\end{prop}
\begin{proof}
As in the proof of Proposition \ref{pro1}, we write:
\begin{multline}
     \kappa_N(\lambda_1, \ldots, \lambda_l)\mathbb{E}_{N}^{\left(a,b\right)}\left[\prod_{j=1}^l x_j^{-\lambda_j/2}\left(1-x_j\right)^{\lambda_j/2}\right]\\ \leq  \text{const } \int_{[0,1]^l} \prod_{j=1}^l x_j^{-\lambda_j/2}\left(1-x_j\right)^{\lambda_j/2} \det\left[K_N^{(a,b)}(x_i,x_j)\right]_{i,j=1}^l dx_1 \ldots dx_l.
     \label{eq:smallint}
\end{multline}
Now, again arguing as in Proposition \ref{pro1} we can bound this integral by:
\begin{equation*}
    \text{const} \cdot \prod_{j=1}^l \int_0^1 x^{-\lambda_j/2}(1-x)^{\lambda_j/2} K_N^{\left(a,b\right)}(x,x) dx.
\end{equation*}
Further, arguing as in Proposition \ref{pro1}, we obtain the bound:
\begin{equation}
    \int_0^1 x^{-\lambda/2}(1-x)^{\lambda/2} K_N^{\left(a,b\right)}(x,x) dx = O\left(N^\lambda \right),
    \label{eq:momentsdct}
\end{equation}
for $\lambda\geq 2$. This proves the desired result since if $l<h$, there exists some $\lambda_j \geq 2$.
\end{proof}
Hence, for $h=2$, upon dividing by $N^2 \log^2 N$ and letting $N\to \infty$, the only (possibly) nonzero limiting contribution will come from the term:
\begin{multline*}
    \mathbb{E}_{N}^{\left(a,b\right)}\left[ x_1^{-1/2}x_2^{-1/2} \left(1-x_1\right)^{1/2}\left(1-x_2\right)^{1/2}\right] \\=\frac{1}{N(N-1)} \int_0^1 \int_0^1 \left[K_N^{\left(a,b\right)}(x,x)K_N^{\left(a,b\right)}(y,y)- \left(K_N^{\left(a,b\right)}(x,y)\right)^2\right]\\ x^{-1/2}y^{-1/2}\left(1-x\right)^{1/2}\left(1-y\right)^{1/2} dx dy.
\end{multline*}
We already know from Proposition \ref{prop1} that:
\begin{equation*}
    \int_0^1 \int_0^1 x^{-1/2}y^{-1/2}\left(1-x\right)^{1/2}\left(1-y\right)^{1/2} K_N^{\left(a,b\right)}(x,x)K_N^{\left(a,b\right)}(y,y)  dx dy \sim \frac{4}{\pi^2} N^2 \log^2(N).
\end{equation*}
Thus, we next prove that the other term gives a contribution of $0$ in the limit, which will give us the desired asymptotics.
\begin{prop}
Let $a\geq b \geq -\frac{1}{2}$, $a \geq 0$. Then,

\begin{equation}
    \lim_{N \to \infty}\frac{1}{N^2 \log^2 N}\int_0^{1} \int_0^{1} x^{-1/2}y^{-1/2}\left(1-x\right)^{1/2}\left(1-y\right)^{1/2} \left(K_N^{\left(a,b\right)}(x,y) \right)^2 dx dy = 0.
\end{equation}
 \label{prop3}
\end{prop}
\begin{proof}
The main idea, again, will be to estimate this integral using Bessel functions. First, note that we can instead consider the integral over $\left[0,\frac{1}{2}\right] \times \left[0,\frac{1}{2}\right]$ since,
\begin{multline*}
    \int_0^{1} \int_{1/2}^{1} x^{-1/2}y^{-1/2}\left(1-x\right)^{1/2}\left(1-y\right)^{1/2} \left(K_N^{\left(a,b\right)}(x,y) \right)^2 dx dy \\ \leq \int_{0}^1 x^{-1/2}(1-x)^{1/2} K_N^{(a,b)}(x,x) dx \int_{1/2}^1 y^{-1/2}(1-y)^{1/2}K_N^{(a,b)}(y,y) dy \\\sim \mathbb{E}_N^{(a,b)}\left[x_1^{-1/2}(1-x_1)^{1/2} \mathbf{1}_{\{x_1\geq \frac{1}{2}\}}\right] N^2 \log N  = O(N^2 \log N),
\end{multline*}
and similarly we get a bound of $O(N^2)$ for the integral over $\left[\frac{1}{2},1\right]\times \left[\frac{1}{2},1\right]$. Now, we observe that,
\begin{equation*}
    \int_0^{1/2}\int_0^{1/2} x^{-1/2}y^{-1/2}\left(1-x\right)^{1/2}\left(1-y\right)^{1/2} \left(K_N^{\left(a,b\right)}(x,y) \right)^2 dx dy = \sum_{j,k=1}^N \left({I}(j,k)\right)^2,
\end{equation*}
where
\begin{equation*}
    I(j,k)= \int_0^{1/2} x^{-1/2}\left(1-x\right)^{1/2}w^{\left(a,b\right)}(x)  P_j^{\left(a,b\right)}(x) P_k^{\left(a,b\right)}(x) dx.
\end{equation*}
Since $I(j,k)$ is symmetric, we may write:
\begin{equation*}
    \sum_{j,k=1}^N I^2(j,k) = 2 \sum_{0\leq k<j \leq N} I^2(j,k) + \sum_{j=1}^N I^2(j,j).
\end{equation*}
Now, we bound the second term using Proposition \ref{prop1}:
\begin{equation*}
    \sum_{j=1}^N I^2(j,j) \leq \text{const }\sum_{j=1}^N \log^2(j) +1  = O\left(N \log^2 N\right).
\end{equation*}
Thus, it is sufficient to consider only the first sum. With this aim, fix some small $\delta\in (0,1)$. Divide the first sum into two sums:
\begin{equation*}
    \mathcal{T}_{N,\delta} + \mathcal{B}_{N,\delta}=\sum_{\substack{j,k=1 \\ \frac{j}{k}\geq 1+\delta}}^N I^2(j,k) + \sum_{\substack{j,k=1 \\ 1< \frac{j}{k}\leq 1+\delta}}^N  I^2(j,k).
\end{equation*}
$\textbf{Claim 1:}$ For fixed $\delta \in (0,1)$, 
\begin{equation*}
 \lim_{N \to \infty} \frac{\mathcal{T}_{N,\delta}}{N^2 \log^2 N} = 0. 
\end{equation*}

To obtain an upper bound on $I(j,k)$, we will again use the estimate \eqref{eq:wongest}. Note that if $\left\{a_l(x)\right\}_{l=1}^\infty$ and $\left\{b_l(x)\right\}_{l=1}^\infty$ are two sequences of functions on $[0,\frac{1}{2}]$ such that:
\begin{equation*}
    \lim_{l \to \infty} \sup_{x \in [0,\frac{1}{2}]}\left|a_l(x)-b_l(x)\right|=0,
\end{equation*}
then there exists an $M>0$ such that,
\begin{multline*}
    \left|a_j(x)a_k(x)-b_j(x)b_k(x)\right| \leq \left|a_j(x)\right| \left|a_k(x) -b_k(x)\right| +\left|b_k(x)\right|\left|a_j(x)-b_j(x)\right|\\ \leq \text{const}\left(\left|a_j(x)\right| +\left|b_k(x)\right|\right) \leq M\left(1 + \left|b_j(x)\right| +\left|b_k(x)\right|\right).
\end{multline*}
Thus, arguing as before, we obtain that there exists a constant $K>0$ such that:
\begin{equation*}
    \sup_{j,k} \left|I(j,k)-L(j,k)\right| \leq K,
\end{equation*}
where we define:
\begin{multline*}
    L(j,k)=\int_0^{1/2} \sqrt{h_j h_k}\frac{\arccos\left(1-2x\right)}{x}\\ \frac{J_{a}\left(\left(j+\frac{a+b+1}{2}\right)\arccos(1-2x)\right)
    J_{a}\left(\left(k+\frac{a+b+1}{2}\right)\arccos(1-2x)\right)}{\left(j+\frac{a+b+1}{2}\right)^{a}\left(k+\frac{a+b+1}{2}\right)^{a}} dx.
\end{multline*}
Thus, for all $j,k$, using the trivial bound $|x|\leq x^2+1$:
\begin{equation}
    I^2(j,k) \leq  C L^2(j,k) + \text{const}.
    \label{eq:ILcomparison}
\end{equation}
Hence, it is enough to show the claim for $I(j,k)$ replaced by $L(j,k)$. Note that making the substitution $\theta=\arccos(1-2x)$ in the definition of $L(j,k)$ gives:
\begin{multline*}
    L(j,k)= \int_0^{\pi/2} \frac{\sqrt{h_j h_k}}{\left(j+\frac{a+b+1}{2}\right)^{a}\left(k+\frac{a+b+1}{2}\right)^{a}} \frac{\theta \sin \theta}{1-\cos \theta} \\J_{a}\left(\left(j+\frac{a+b+1}{2}\right)\theta\right)
    J_{a}\left(\left(k+\frac{a+b+1}{2}\right)\theta\right) d\theta.
\end{multline*}
Now, we divide the integral into two contributions: $ (0,\frac{1}{k})$ and $(\frac{1}{k},\frac{\pi}{2})$. Note that since Bessel functions are bounded (by a constant depending on the parameter of the function), the contribution on $(0,\frac{1}{k})$ is bounded by:
\begin{equation*}
    \text{const} \cdot \frac{1}{k} \frac{\sqrt{h_j h_k}}{\left(j+\frac{a+b+1}{2}\right)^{a}\left(k+\frac{a+b+1}{2}\right)^{a}}  \leq \text{const}  \sqrt{\frac{j}{k}},
\end{equation*}
Now, to get asymptotics on the other contribution, we use \eqref{eq:besselasymptotics}. To be more precise, since on $\theta \in \left(\frac{1}{k},\frac{\pi}{2}\right)$; $\left(j+\frac{a+b+1}{2}\right)\theta $ and $\left(k+\frac{a+b+1}{2}\right)\theta$ are bounded away from zero uniformly in $j,k$; we see that there exists a bounded function $f : [1,\infty) \to \mathbb{R}$ such that:
\begin{multline*}
    J_{a}\left(\left(j+\frac{a+b+1}{2}\right)\theta\right)\\ = \sqrt{\frac{2}{\pi\left(j+\frac{a+b+1}{2}\right)\theta }} \cos\left(\left(j+\frac{a+b+1}{2}\right)\theta- \frac{\pi (a+\frac{1}{2})}{2}\right) + \frac{f\left(\left(j+\frac{a+b+1}{2}\right)\theta\right)}{\left(j+\frac{a+b+1}{2}\right)^\frac{3}{2}\theta^{\frac{3}{2}}},
\end{multline*}
and similarly with $j$ replaced by $k$.
Using this, we obtain that the contribution on $(\frac{1}{k}, \frac{\pi}{2})$ is given as (since the terms coming from the errors are integrable):
\begin{multline*}
    M(j,k):= \int_{1/k}^{\pi/2}g(\theta) \frac{2\sqrt{h_j h_k}}{\pi\left(j+\frac{a+b+1}{2}\right)^{(a+\frac{1}{2})}\left(k+\frac{a+b+1}{2}\right)^{(a+\frac{1}{2})}} \\ \frac{\cos\left(\left(j+\frac{a+b+1}{2}\right)\theta- \frac{\pi (a+\frac{1}{2})}{2}\right)\cos\left(\left(k+\frac{a+b+1}{2}\right)\theta- \frac{\pi (a+\frac{1}{2})}{2}\right)}{\theta} d\theta + O(1) \\=O(1)  \int_{1/k}^{\pi/2}g(\theta)  \frac{\cos\left(\left(j+\frac{a+b+1}{2}\right)\theta- \frac{\pi (a+\frac{1}{2})}{2}\right)\cos\left(\left(k+\frac{a+b+1}{2}\right)\theta- \frac{\pi (a+\frac{1}{2})}{2}\right)}{\theta} d\theta + O(1).
\end{multline*}
where we recall that we denote:
\begin{equation*}
    g(\theta)=\frac{\theta \sin \theta}{1-\cos \theta}.
\end{equation*}
Now, we divide the integral into two integrals via the identity
\begin{equation}
    \cos(a)\cos(b) =\frac{1}{2}\left(\cos(a-b)+\cos(a+b)\right).
    \label{eq:cosidentity}
\end{equation}
Then, the first integral becomes (up to multiplication by a bounded quantity):
\begin{equation*}
    \int_{1/k}^{\pi/2} \frac{g(\theta)\cos \left((j-k)\theta\right)}{\theta} d\theta = \int_{1}^{\frac{\pi k}{2}} \frac{\cos\left(\left(\frac{j}{k}-1\right)u\right) g\left(\frac{u}{k}\right)}{u} du,
\end{equation*}
which can be seen to be bounded for all $j,k$ such that $\frac{j}{k}-1\geq \delta>0$, using integration by parts. As for the second integral that comes from \eqref{eq:cosidentity},
\begin{multline*}
     \int_{1/k}^{\pi/2} \frac{g(\theta)\cos \left((j+k+a+b+1)\theta-\pi (a+\frac{1}{2})\right)}{\theta} d\theta \\ = \int_1^{\frac{\pi k}{2}} \frac{g\left(\frac{u}{k}\right)\cos \left(\left(\frac{j+a+b+1}{k}+1\right)u-\pi (a+\frac{1}{2})\right)}{u} du = O(1),
\end{multline*}
which can be seen again via integrating by parts. In particular, using the trivial identity
$(x+y)^2 \leq 2(x^2+y^2)$, there exists a constant $C>0$ such that for all $j,k$ such that $\frac{j}{k}\geq 1+\delta$:
\begin{equation*}
    L^2(j,k) \leq C \frac{j}{k}.
\end{equation*}
Noting that
\begin{equation*}
    \sum_{\substack{j,k=1 \\ \frac{j}{k}\geq 1+\delta}}^N \frac{j}{k} \leq \left(\sum_{k=1}^N\frac{1}{k}\right)\left(\sum_{j=1}^N j\right) = O( N^2 \log(N)),
\end{equation*}
the claim follows.\\
$\textbf{Claim 2}:$ There exists a constant $C>0$ independent of $\delta$ such that 
\begin{equation*}
  \limsup_{N \to \infty} \frac{\mathcal{B}_{N,\delta}}{N^2 \log^2 N}\leq C \delta.
\end{equation*}
Note that once again, it is sufficient to prove the claim for $I(j,k)$ replaced by $L(j,k)$. To prove this bound, we start again by dividing the integral  defining $L(j,k)$ into two contributions: $\left(0,\frac{1}{k}\right)$ and $\left(\frac{1}{k},\frac{\pi}{2}\right)$. Exactly as before, we get the bound,
\begin{equation*}
    C \sqrt{\frac{j}{k}} \leq C \sqrt{1+\delta} \leq 2C,
\end{equation*}
on the integral over $\left(0,\frac{1}{k}\right)$, where $C$ only depends on the parameter of the Bessel function and hence does not depend on $\delta$. As for the other integral, we use the estimate,
\begin{equation*}
   \left|J_a(x)\right| \leq \frac{C}{\sqrt{x}},
\end{equation*}
to show that the integral over $\left(\frac{1}{k},\frac{\pi}{2}\right)$ is bounded by 
\begin{equation*}
    \text{const} \int_{1/k}^{\pi/2} \frac{1}{u} du = O(\log k ).
\end{equation*}
Thus, we get a constant $C>0$ independent of $\delta$ such that
\begin{equation*}
   \mathcal{B}_{N,\delta} \leq C \left( \sum_{\substack{j,k=1 \\ 1< \frac{j}{k}\leq 1+\delta}}^N  \log^2(k) + 1  \right).
\end{equation*}
Now, note that for each $k$, $\log^2(k)+1$ appears as many times as the number of integers $j\leq N$ such that $k < j \leq k(1+\delta)$. But this is at most $\delta k $, and hence,
\begin{equation*}
    \sum_{\substack{j,k=1 \\ 1< \frac{j}{k}\leq 1+\delta}}^N  \log^2(k) +1 \leq \delta \sum_{k=1}^N k (\log^2 (k)+1) \leq \delta N^2 \log^2 N + \delta N^2,
\end{equation*}
which proves the claim.
Combining the two claims, we see that:
\begin{equation*}
    \limsup_{N\to \infty} \frac{1}{N^2\log^2 N} \sum_{j,k=1}^N I^2(j,k) =  \limsup_{N \to \infty} \left(\mathcal{T}_{N,\delta}+\mathcal{B}_{N,\delta}\right) \leq \limsup_{N \to \infty} \mathcal{T}_{N,\delta}+ \limsup_{N\to \infty}\mathcal{B}_{N,\delta}\leq C \delta.
\end{equation*}
Since $C$ is independent of $\delta$ and $\delta \in (0,1)$ was arbitrary the result follows.
\end{proof}
\begin{proof}[Proof of Theorem \ref{mainthm}]
Combining Propositions \ref{prop1}, \ref{prop2}, \ref{prop3}, one can see that for $h=1,2$,
\begin{equation*}
    \lim_{N \to \infty} \mathbb{E}_{N}^{\left(a,b\right)} \left[\left(\mathcal{Z}_N\right)^h\right] = \left(\frac{2}{\pi}\right)^h.
\end{equation*}
In particular:
\begin{equation*}
     \lim_{N \to \infty} \mathbb{E}_{N}^{\left(a,b\right)} \left[\left(\mathcal{Z}_N-\frac{2}{\pi}\right)^2\right] =\lim_{N \to \infty}\left[ \mathbb{E}_{N}^{\left(a,b\right)} \left[\left(\mathcal{Z}_N\right)^2\right] - \frac{4}{\pi} \mathbb{E}_{N}^{\left(a,b\right)} \left[\mathcal{Z}_N\right] + \frac{4}{\pi^2}\right]=0.
\end{equation*}
Hence, we have that $\mathcal{Z}_N$ converges, in distribution, to $\frac{2}{\pi}$, and hence by Skorokhod's Representation Theorem we may assume without loss of generality that the convergence is in almost sure sense. Moreover, expanding as in \eqref{eq:combsum5} and applying Propositions $\ref{prop1}, \ref{prop2}$ and $\ref{prop3}$, one can infer,
\begin{equation*}
    \sup_{N \geq 1} \left\{\mathbb{E}_{N}^{\left(a,b\right)} \left[\left(\mathcal{Z}_N\right)^
{\alpha(2a+2)}\right] \right\}< \infty,
\end{equation*}
which means the sequence of random variables
\begin{equation*}
    \left\{\left(\mathcal{Z}_N\right)^h\right\}_{N \geq 1},
\end{equation*}
sampled according to $\mu_N^{(a,b)}$
are uniformly integrable. Combining the uniform integrability of this sequence with the assumed almost sure convergence, we get the desired result.
\end{proof}

\bibliographystyle{siam}
\bibliography{references.bib}

 \noindent
{\sc  Fine Hall, 304 Washington Rd, Princeton, NJ 08544, USA.}\newline
\href{mailto:magunes@princeton.edu}{\small magunes@princeton.edu}
\end{document}